\documentclass[runningheads]{llncs}
\usepackage{graphicx}
\usepackage{amsmath}
\usepackage{amssymb}
\usepackage{amsfonts}
\usepackage{mathtools}
\usepackage[english]{babel}
\usepackage{enumerate}
\usepackage{url}
\usepackage{color}
\usepackage{subfigure}
\usepackage[normalem]{ulem}
\usepackage{lipsum}
\usepackage{authblk}
\usepackage{hyperref}
\usepackage{cite}
\usepackage{todonotes}
\usepackage{mathabx}
\usepackage{tikz}
\usepackage{pgfplots}

\newcommand{\bR}{\mathbb R}
\newcommand{\bC}{\mathbb C}
\newcommand{\bM}{\mathbb M}
\newcommand{\bB}{\mathbb B}

\newcommand{\bZ}{\mathbb Z}

\newcommand{\cM}{{\cal M}}

\newtheorem{DE}{Definition}[section]
\newtheorem{AS}[DE]{Assumption}

\newtheorem{LE}[DE]{Lemma}

\newtheorem{RE}[DE]{Remark}
\newtheorem{THM}[DE]{Theorem}
\newtheorem{CO}[DE]{Corollary}

\begin{document}
\title{Solving convex QPs with structured sparsity under indicator conditions\thanks{Funded by ARPA-e and AFOSR.}}
\titlerunning{Structured indicators}
%
\author{Daniel Bienstock\inst{1} \and
Tongtong Chen\inst{2}}
\authorrunning{Bienstock, Chen.}
%
\institute{IEOR and Applied Mathematics, Columbia University \and
IEOR, Columbia University \\
\email{\{dano,tc3240\}@columbia.edu}}
\maketitle              
\begin{abstract} 
We study convex optimization problems where disjoint blocks of variables are controlled by binary indicator variables that are also subject to conditions, e.g., cardinality. Several classes of important examples can be formulated in such a way that both the objective and the constraints are separable convex quadratics.  We describe a family of polynomial-time approximation algorithms and negative complexity results. 
\end{abstract}

\section{Introduction} 
We consider optimization problems on $n$ real variables $x$ that are partitioned into a set $B$ of \textit{blocks}; the $i^{th}$ block is denoted $B_i$ and we write $x{[i]}$ for the subvector of $x$ restricted to $B_i$.  Corresponding to each block there is a binary variable $z_i$. The problems have the following structure (additional notation below):
\begin{subequations}\label{QuadBlockInd}
\begin{align}
[QIB]: &\ \min \ \sum_{i = 1}^{|B|} x[i]^T Q^0_i x[i] \ + \ \sum_{j = 1}^n d^0_j x^2_j \ + \ \sum_{j = 1}^n c^0_j x  \ + \ \sum_{i = 1}^{|B|} v^0_i z_i\\
    \text{s.t.} \quad & \sum_{i = 1}^{|B|} x[i]^T Q^r_i x[i] \ + \ \sum_{j = 1}^n d^r_j x^2_j \ + \ \sum_{j = 1}^n c^r_j x_j \ + \ \sum_{i = 1}^{|B|} v^r_i z_i  \ \le \ p_r, \quad 1 \le r \le m \label{eq:mainbody}\\        
        & \ell_j \ \le \ x_j \le u_j, \ \text{for} \ 1 \le j \le n, \quad z_i \in \{0,1\} \ \text{for} \ 1 \le i \le |B|,\\
    &  \text{for $1 \le i \le |B|$, if $z_i = 0$ then $x_j = 0$ for all $j \in B_i$}, \label{eq:block}
    \end{align}
\end{subequations}
We assume that every inequality \eqref{eq:mainbody} falls into one of two types:
\begin{enumerate}
\item \textit{Mixed-integer} constraints; those with $Q^r_i \neq 0$ or    $d^r\neq 0$ or $c^r\neq 0$. We assume  $Q^r_i \succeq 0$ and  $d^r \ge 0$.
\item \textit{Combinatorial} constraints.  Here $Q^r_i = 0$ for all $i$ and $d^r_j = c^r_j = 0$ for all $j$, and $v^r_i \in \{0,1,-1\} \, \forall i$.
\end{enumerate}
 Constraints \eqref{eq:block} enforce the indicator relationship over blocks of variables. Motivation for this type of model is given below; here we point out that combinatorial constraints include cardinality constraints as a special case;    there is an abundant literature for this particular feature where the blocks are singletons, in which case problem \eqref{QuadBlockInd} is (strongly) NP-hard\footnote{Our results extend to the simple generalization of combinatorial constraints where the $v^r_i$ can take integral values in some fixed range.}.

In this paper we study the case where constraints \eqref{eq:mainbody} exhibit a (new) parameterized structured sparsity condition that arises in a number of important problems, when appropriately reformulated.  We present an algorithm that under such conditions runs in polynomial time and yields solutions with provable approximation guarantees.  We posit that this theoretical development helps explain why problems of the above form, while NP-hard, can (empirically) be effectively approximated.

As a key ingredient in our approach, we introduce a representation of the sparsity structure of a problem of the above form:

\begin{DE} \label{def:incidence}  {\bf Constraint-block incidence.}  Given a problem of type QIB, we say that a constraint $r$ and a block $B_i$ are \textit{incident} with each other if $Q^r_i \neq 0$, or $\sum_{j \in B_i} (d^r_j + | c^r_j|) > 0$ or $| v^r_i | > 0$, in other words, if constraint $r$ includes at least one term related to block $B_i$. The set of blocks incident with a given constraint $r$ is denoted by $\bB^r$.  \end{DE} 
\begin{DE} {\bf Constraint-block intersection graph.} Given a problem of type QIB, its \textit{constraint-block intersection graph} is the graph
whose vertex-set is the set of constraints \eqref{eq:mainbody} and there is an edge between two constraints, $r$ and $r'$, if there is a block $B_i$ that both $r$ and $r'$ are incident with. \end{DE}

We will capture the complexity of a problem of type QIB through the structure of its constraint-block intersection graph; in particular through
 its \textit{treewidth} or \textit{pathwidth} (definitions given later).  For a matrix $A$ we write $ \|A\|_1 = \sum_{h,k} |a_{hk}|$.  For simplicity of language, we assume:
 \begin{AS} \label{as:scaling} (i) For $1 \le j \le n$, $\ell_j = 0$ and $u_j = 1$. (ii) Every mixed-integer constraint $r$ satisfies
 \begin{align} \label{eq:scaling}& \sum_{i = 1}^{|B|} \|Q^r_i\|_1 \ + \  \sum_{j = 1}^n ( d^r_j + | c^r_j |)  \, + \, \sum_{i = 1}^{|B|} |v^r_i| \ = \ 1.
\end{align}
\end{AS}
Assumption \ref{as:scaling} is obtained via a change of variables and constraint-scaling, and will be assumed in the mathematical analyses below:

\begin{THM}\label{thm:mainone} Let $0 < \epsilon < 1$ and let $\omega$ denote the treewidth of the constraint-block intersection graph. Let $\kappa$ denote the largest support of any combinatorial constraint, and let $m_c$ denote the number of combinatorial constraints. There is an algorithm that runs in time $O( (n + m)(2/\epsilon)^{2\omega + 2} \kappa^{2 \min\{ m_c, \omega + 1 \}}   )$, plus the time needed to solve $O(2|B|/\epsilon)$ convex QCQPs, each a restriction of \eqref{QuadBlockInd} to a single block (hence with $\le \omega + 1$ rows) with appropriate right-hand side vectors; the algorithm either proves infeasibility or computes a 
vector that is (a) superoptimal, (b) feasible for the combinatorial constraints and (c) has maximum infeasibility $\le \max_r |\bB^r| \epsilon$ for the mixed-integer constraints.
\end{THM}
In contrast to this result we have the following:
\begin{LE}\label{le:width2} Testing feasibility for a linear system $Ax = b$ where (a) each variable $x_j$ constitutes a block and (b) there is a cardinality constraint on the blocks, i.e., $\| x \|_0 \le K$ for some $K$, is NP-hard even if the constraint-block graph has pathwidth $\le 2$. \end{LE}
\begin{LE}\label{le:2rowcard} Problem QIB is NP-hard even if there are just two mixed-integer constraints, both being linear equations on the $x$, and there is a cardinality constraint on the blocks (singletons, all), and the objective is  separable. \end{LE}

There is an abundant literature on optimization under structured sparsity, in particular treewidth.  In particular, \cite{BienstockMunoz2018} studies mixed-integer polynomial optimization problems under a different treewidth setup.  Problems of type QIB can have arbitrarily large blocks, and, as a result, the algorithm in \cite{BienstockMunoz2018} may have exponential runtime.  Also, there is an even larger literature on purely combinatorial problems under constant treewidth (see, e.g., \cite{bodlaender}).  In this context, one can use problems of type QIB, without mixed-integer constraints, to encode combinatorial problems.  For example, Theorem \ref{thm:mainone} states that we can solve, in polynomial time, a vertex packing problem on a graph whose line graph has bounded treewidth. 

This paper is organized as follows.  Section \ref{sec:twdef} introduces tree-decompositions and related concepts. Section \ref{sec:appl} describes the application of our results to specific examples of problem QIB. Section \ref{sec:algorithm} presents our algorithm. Section \ref{sec:nphard} shows that even highly restricted versions of problem QIB are already NP-hard. Section \ref{sec:glossary} contains a glossary of technical terms. 

\subsection{Tree-decompositions,  treewidth and pathwidth}\label{sec:twdef}

For future reference we provide some basic definitions here.

\begin{DE} {\bf Tree-decomposition}.  A \textit{tree-decomposition} of a graph $G$ consists of a pair $(T, X)$ where $T$ is a tree\footnote{Not a subtree of $G$; i.e., an abstract tree.} and $X = \{ X_t \, : \, t \in V(T)\}$ is a family of subsets of $V(G)$ with the properties that: (1) $\forall u \in V(G)$ the set of $t \in V(T)$ s.t. $u \in X_t$ is a nonempty subtree of $T$, and (2) $\forall \{u,v\} \in E(G)$, both $u$ and $v$ are contained in some common set $X_t$.   The {\bf width} of the tree-decomposition is $\max_{t \in V(T)} |X_t| - 1$.\end{DE}
Comment: the sets $X_t$ are referred to as the ``bags'' of the tree-decomposition.  

\begin{DE} \label{def:widths}{\bf Treewidth}.  The treewidth of a graph $G$ is the minimum width of any tree-decomposition of $G$.  If we constrain the tree-decompositions to paths, the minimum width is termed the pathwidth. \end{DE} 
\section{Applications and Examples}\label{sec:appl}
Here we describe a set of problems that can be rendered into the form QIB with low treewidth, thus obtaining polynomial-time approximation algorithms.  
\subsection{Cardinality-constrained portfolio optimization}
There is a large literature on problems of the following form; see, e.g., \cite{Bienstock1996}, \cite{BertsimasRyan}:
\begin{subequations}\label{eq:portopt}
\begin{align}
 \min & \ x^T Q x \ - \ \mu^T x \\
    \text{s.t.} & \ Ax \le b \label{eq:body} \\ 
      & \ 0 \le x_j \le u_j z_j, \ z_j \in \{0,1\}, \ 1 \le j \le n \\
      & \ \sum_{j = 1}^n z_j \ \le \ N. \label{eq:thecard}
\end{align}
\end{subequations}
Here, $Q \succeq  0$ is a positive multiple of a covariance matrix, and $\mu$ is a vector.  The matrix $A$, often, has \textit{few rows} compared to $n$, though dense.  In the textbook form of the problem, \eqref{eq:body} just consists of the equation $\sum_j x_j = 1$.  Finally, \eqref{eq:thecard} is a cardinality constraint.  It is known that this problem is NP-hard even under very restrictive conditions.  If the number of rows in $A$ is part of the input the problem is strongly NP-hard.

There is one more feature that is critical in real-world instances of this problem. Namely, the quadratic part of the objective is rendered as 
\begin{align}\label{eq:lowrankplussparse}
& x^T Q x \ = \ x^T\tilde Q x \ + \sum_j d_j x^2_j
\end{align}
Here $\tilde Q \succeq 0$ is \textit{of low rank} and $d_j \ge 0$ for all $j$.   This is an example of a ``low rank plus sparse'' structure and we can take advantage of this feature to reduce to a case of problem QIB by diagonalizing the quadratic $x^T \tilde Q x$.  This yields formulation \eqref{eq:portoptsparse}, discussed immediately below. 
\begin{subequations}\label{eq:portoptsparse}
\begin{align}
 \min & \ \sum_{h = 1}^H \lambda_h y^{2}_h  + \ \sum_{j = 1}^n d_j x^2_j \ - \ \mu^T x \\
    \text{s.t.} & \ Ax \le b \label{eq:body2} \\ 
    & \ (v^h)^T x \ = \ y_h, \ 1 \le h \le H \label{eq:eigen}\\
      & \ 0 \le x_j \le u_j, \ \ z_j \in \{0,1\}, \ \ x_j = 0 \ \text{if} \  z_j = 0, \ \ 1 \le j \le n \label{eq:contbin} \\ 
            & \ \sum_{j = 1}^n z_j \ \le \ N. \label{eq:thecard2}
\end{align}
\end{subequations}
Here, $H$ denotes the rank of $\tilde Q$, and, for $1 \le h \le H$,   $\lambda_h$ and $v^h$ denote the $h^{th}$ eigenvalue and 
eigenvector of $\tilde Q$, respectively.
We assume that this eigenvector-eigenvalue decomposition is provided as an input to our algorithm, thus bypassing discussions about polynomial-time computability of suitable approximations thereof -- however we note that, indeed, the  decompositions are, often, provided as inputs to optimization algorithms.  

With this proviso, 
\eqref{eq:portoptsparse} is an equivalent reformulation of the problem.
Here,  \eqref{eq:eigen} enforces the eigenvector decomposition, i.e., $x^T \tilde Q x = \sum_{h = 1}^H \lambda_h y_h^2$. 
This problem is almost of the form QIB, except that the $y$ variables do not belong to any block, and that they are not explicitly  bounded.  The latter issue is easily circumvented using \eqref{eq:eigen} and the upper bounds on the 
$x$ variables.  The former is handled by adding, for each $1 \le h \le H$, a new block $\{y_h\}$ and a corresponding new binary variable which does not appear in any (other) constraint.\\ 

\noindent {\bf Note:} Let $R$ be the number of rows of $A$. We stress that, often, $R = O(1)$ can be quite small.  Hence the treewidth, $\omega$, of the constraint-block intersection graph \eqref{eq:body2}-\eqref{eq:eigen} is also small. The number of constraints \eqref{eq:eigen}, \eqref{eq:thecard2} is $H +1$.  Hence the treewidth of the constraint-block intersection graph for the entire problem \eqref{eq:portoptsparse} is at most $\omega + H + 1 \le R + H$ (by adding, to each ``bag'' of the tree-decomposition for system \eqref{eq:body2}, the constraints \eqref{eq:eigen} and \eqref{eq:thecard2}).  Finally the number of blocks is $n + H$ 
and the number of continuous variables is $n + H$.     We also note that in constraint \eqref{eq:contbin} we rely on the language of constraint \eqref{eq:block} of problem QIB, i.e., $x_j \le u_j z_j$ is \textit{implicit}. Using Theorem \ref{thm:mainone} (with singletons as blocks) we obtain: 
\begin{LE}\label{le:portopt1} For each $0 < \epsilon < 1$, there is an algorithm that either proves infeasibility or computes an $\epsilon-$feasible and superoptimal solution in $O( n(2/\epsilon)^{2\omega + 2H + 2} n^{2}   )$    time; the output
of the algorithm satisfies the cardinality constraint.
\end{LE}
\vspace{-25pt}
\subsubsection{Long-short problems with orthogonality constraints.}
In ``long-short'' portfolio optimization problems the variables $x_j$ can be negative and there may be constraints involving the positive or negative parts of the $x_j$.  Moreover there may be cardinality constraints on the positive $x_j$, on the negative $x_j$ and on the nonzero $x_j$.  Instead of problem \eqref{eq:portoptsparse} we obtain
\begin{subequations}\label{eq:portoptsparseLS}
\begin{align}
 \min & \ \sum_{h = 1}^H \lambda_h y^{2}_h  + \ \sum_{j = 1}^n d_j ((x^+)^2_j + (x^-)^2_j ) \ - \ \mu^T x \\
    \text{s.t.} & \ A^+x^+  + A^-x^- \le b \label{eq:bodyLS} \\ 
    & \ (v^h)^{T} (x^+ - x^-) \ = \ y_h, \ 1 \le h \le H \label{eq:eigenLS}\\
      & \ 0 \le x^+_j \le u^+_j, \ z^+_j \in \{0,1\},  \ x^+_j = 0 \ \text{if} \  z^+_j = 0, \ 1 \le j \le n \\ 
      & \ 0 \le x^-_j \le u^-_j, \ z^-_j \in \{0,1\},  \ x^-_j = 0 \ \text{if} \  z^-_j = 0, \ 1 \le j \le n \\ 
            & \ z^+_j \, + \, z^-_j \ \le \ 1, \  1 \le j \le n, \label{eq:ortho} \\
            & \ \sum_{j = 1}^n z^+_j \ \le \ N^+, \ \sum_{j = 1}^n z^-_j \ \le \ N^-, \  \sum_{j = 1}^n (z^+_j  + z^-_j) \ \le \ N. \label{eq:thecardall}\end{align}
\end{subequations}
Here, the variables $x_j$ have been replaced by pairs $x^+_j, x^-_j$ representing (resp.) their positive and negative parts (the equation $x_j = x^+_j - x^-_j$ is implicit).  Each of these two
variables constitutes a block, controlled respectively by the binary variables $z^+_j$ and $z^-_j$.  Constraint \eqref{eq:ortho} (a combinatorial constraint, in our language) forces  $x^+_j x^-_j = 0$ for all $j$; an orthogonality constraint.  Finally \eqref{eq:thecardall} describes cardinality constraints on the positive parts, the negative parts, and $|x|$. 

Suppose that $\rho$ is the treewidth of the constraint-block intersection graph for \eqref{eq:bodyLS}.  Then the corresponding treewidth for the combination of \eqref{eq:bodyLS}, \eqref{eq:eigenLS} and \eqref{eq:thecardall} is at most $\rho + H + 3$ following a similar argument as above, i.e., having constraints \eqref{eq:eigenLS} and \eqref{eq:thecardall} in all bags of the tree-decomposition.  Now we can modify said tree-decomposition by add $n$ new leaves, each with constraints \eqref{eq:eigenLS}, \eqref{eq:thecardall} and \eqref{eq:ortho}.  In summary, we obtain
\begin{LE}\label{le:portopt2} For each $0 < \epsilon < 1$, we can compute an $\epsilon-$feasible and superoptimal solution to problem \eqref{eq:portoptsparseLS} in polynomial time for each fixed $\rho$ and $H$.  The output of the algorithm satisfies \eqref{eq:ortho} and \eqref{eq:thecardall}.
\end{LE}

\noindent {\bf Thresholding.} There is one additional type of constraint that can be modeled without increase of complexity; $x^+_j \ge \delta^+_j z^+_j$ (and similarly with negative parts) where $\delta^+_j > 0$ is a threshold.  Our dynamic-programming algorithm given below will appropriately enumerate $z^+_j \in \{0, 1\}$ at which point the thresholding constraint becomes a linear inequality and is enforced without violation.

\subsection{Convex quadratic optimization problems with banded matrices and indicator variables}
Here we consider a problem of the following form studied in
\cite{Gomez2024}
\begin{align} \label{eq:gomez}
    h^* \ \doteq \ \min \quad &  d^T z \ + \ c^T x \ + \ x^T Q x\\
    \text{s.t.}\quad &  x \in \bR^n, \ z \in \{0,1\}^n, \ x_j =0 \ \text{if} \ z_j = 0, \ \forall j,\nonumber
  \end{align}
where $Q \succ 0$ is a \textit{banded matrix with bandwidth} $\le k$, i.e., $Q_{ij} = 0$ if $|i - j| > k$.  The very interesting algorithm in \cite{Gomez2024}, given $0 < \epsilon < 1$, computes a value $h$
with $|h - h^*| < \epsilon$, while performing at most 
$$ n \left(\frac{C \| d \|_{\infty}^2n}{\epsilon} \right)^{1 + \frac{k}{| \log \gamma|}}$$
operations, where $\gamma = \frac{\sqrt{\text{cond}(Q)} - 1}{\sqrt{\text{cond}(Q)} + 1}$ (cond($Q$) is the ratio of the largest to the smallest eigenvalue of $Q$) and $C$ is a constant proportional to the largest entry of $Q$ and $(1 - \gamma^{1/k})^{-1}$ \footnote{The algorithm may compute square roots; however we expect that an approximate version with similar attributes applies under the bit model of computing.}.

To apply our framework, note that any vector $x$ (for example, $(1, \ldots, 1)^T$) is feasible and yields an upper bound $\hat U$ on the optimal objective value. Using this upper bound, and since $Q \succ 0$ we can compute (in polynomial time) an upper bound on the $|x_j|$.  Finally, using the $LDL^T$ decomposition of $Q$ \cite{GolubVanLoan}, scaling and shifting the $x$ variables, we can rewrite the problem as
\begin{align} \label{eq:gomez2}
    \min \quad &  d^T z \ + \ c^T x \ + \ \sum_{j = 1}^n d_j y_j^2\ \ + \ f^T y \\
    \text{s.t.}\quad &  Lx - y \ = \ b, \label{eq:Lrows}\\
    \quad &  x \in [0,1]^n, \ y \in [0,1]^n, \ z \in \{0,1\}^n, \ x_j =0 \ \text{if} \ z_j = 0, \ \forall j,\nonumber
  \end{align}
where $d \ge 0$, $b \in \bR^n$ and $f \in \bR^n$.  The $L$ matrix is lower-triangular and also of bandwidth $\le k$, and, by scaling, we may assume that each row of this matrix has $1$-norm equal to $1$.  Formulation \eqref{eq:gomez2} is an example of QIB with $n$ mixed-integer constraints (linear equations, all) and no combinatorial constraints, where the treewidth is at most $k - 1$. 

Hence our algorithm, applied to this problem, has runtime $O( n(2/\epsilon)^{2k} )$.  The algorithm yields a superoptimal solution with infeasibility (over constraints) \eqref{eq:Lrows} at most $(n+1)\epsilon$. The infeasibilities arise in constraints \eqref{eq:Lrows} and thus, by adjusting the $y$ values, we obtain a feasible solution whose value exceeds the lower bound our algorithm yields by at most $O( \sum_j(d_j + |f_j|) (n \epsilon + n^2 \epsilon^2))$ .  This fact can be combined with a binary search procedure that yields, in time polynomial in the data and $\epsilon^{-1}$ a lower and an upper bound for the value of the problem that differ by at most $\epsilon$.
\subsection{Truss topology design problems}
Here we consider problems of type QIB where every constraint is a linear equation \cite{zhourosvany}:
\begin{align} \label{eq:ZR}
    \min \quad &  \sum_{j = 1}^n x^2_j  \\
    \text{s.t.}\quad &  Ax \ = \ b, \ \ \forall \ \text{block} \ B_i, \ \forall \ j \in B_i: \quad  x_j = 0 \ \text{if $z_i = 0$}, \nonumber \\
    \quad & z_i \in \{0,1\}, \ 1 \le i \le |B|, \ \ \sum_i z_i \le N. \nonumber
  \end{align}
 Note that there are no explicit bounds on the $x$ variables; however such bounds, of polynomial size, can be easily furnished using standard polyhedral arguments \cite{Schrijver}.  In Lemma \ref{le:width2linear} we show that this problem is NP-hard even if $A$ has just two rows and all blocks are singletons.  Moreover, testing feasibility is NP-hard even if the constraint-block intersection graph has pathwidth two (Lemma \ref{le:staircase}).   However, as per the main result in this paper, if the constraint-block intersection graph has bounded treewidth then we can approximately solve the problem in polynomial time.
 
 Problem QPLIB$\_$5925 of the QP Library \cite{qplib}, with $n = 1200$, $|B| = 100$ and $N = 75$, is noteworthy. It is formulated using the (very effective) perspective reformulation \cite{perspective} to enforce the binary variable/block relationship. The pathtwidth in this case is $16$; modern MINLP solvers (such as Gurobi) are able to rapidly obtain solutions within $3 \%$ of optimality but require days of to prove optimality.  
 
\section{Algorithm}
\label{sec:algorithm}
In what follows, the set of mixed-integer constraints incident with a block $B_i$ will be denoted by $\cM_i$.  
\begin{DE} Consider a block $B_i$.  A {\bf sketch} for $B_i$ is given by a pair $( \delta, (k_r \, : \, r \in \cM_i))$, where $\delta \in \{0,1\}$, and  $k_r \in \bZ$, $-\lceil \epsilon ^{-1} \rceil \, \le \, k_r \, \le \, \lceil \epsilon^{-1} \rceil - 1, \, \forall \, r \in \cM_i$. The {\bf value} of the sketch equals:
\begin{subequations}\label{sketchr}
\begin{align}
 &\hspace{2em}\min \hspace{2em} x[i]^T Q^0_i x[i]  \ + \ \sum_{j \in B_i} d^0_j x^2_j \ + \ \sum_{j \in B_i} c^0_j x  \ + \  v^0_i z_i\\
    \text{s.t.} \ & x[i]^T Q^r_i x[i]  \, + \, \sum_{j \in B_i} d^r_j x^2_j \, + \, \sum_{j \in B_i} c^r_j x_j \, + \,  v^r_i z_i  \, \le \,  \epsilon (k_r + 1), \ \forall \,  r \in \cM_i \label{eq:mixedr}\\              
      & z_i \ = \ \delta, \label{eq:fixedzi}\\
        & 0 \ \le \ x_j \le u_j , \ \text{for} \ j \in B_i ,\\
            &   \text{if $\delta = 0$, $x_j = 0$ for all $j \in B_i$}. \label{eq:blocki}
\end{align}
\end{subequations}
If the problem is feasible, we say that the sketch is feasible and that its {\bf certificate} is given by an optimal solution vector.  If $\delta = 1$ the sketch
is {\bf non-zero}.
\end{DE}

\begin{RE}  We have chosen the above, somewhat redundant definition of a sketch in order to simplify the notation in the dynamic programming algorithm given below. Note that \eqref{sketchr} is a convex QCQP with $\le \omega + 1$ constraints. \end{RE}

\noindent Let us also define:
\begin{DE}\label{def:sketchensemble} A \textit{feasible sketch ensemble} consists of a feasible sketch  $( \delta^i, (k^i_r \, : \, r \in \cM_i))$ for each $i \in B$, such that:
\begin{itemize}
\item [(a)]  For each mixed-integer constraint $r$,  $\sum_{i \in B} \epsilon (k^i_r + 1) \, \le \, b^r + \epsilon | \bB^r |$.
\item [(b)] For each combinatorial constraint $r$, $\sum_{i \in B} v^r_i \delta^i \, \le \, b^r$.
\end{itemize}
The \textit{value} of the solution sketch is the sum of the values of sketches $( \delta^i, (k^i_r \, : \, r \in \cM_i))$ over $B_i \in B$. The {\bf certificate} for the ensemble is the 
vector $(\hat x, \hat z)$ with $\hat x \in \bR^n$ and $\hat z \in \{0,1\}^{|B|}$ such that for each $i \in B$, the restriction of $\hat x$ to $B_i$ equals the certificate for the sketch $( \delta^i, (k^i_r \, : \, r \in \cM_i))$;   furthermore we set $\hat z_i = \delta^i$.
\end{DE}

\noindent Note that while we only ask for approximate feasibility over the mixed-integer constraints, we do require actual feasibility over the combinatorial constraints.  Moreover, from Definition \eqref{def:sketchensemble} we have:
\begin{LE}\label{lem:ensembletosol} Let $(\hat x, \hat z)$ be the certificate for a feasible sketch ensemble, and let $\hat v$ be the value of the sketch.  Then $\hat z$ satisfies all combinatorial constraints, and the maximum infeasibility of  $(\hat x, \hat z)$ over mixed-integer constraints is
at most  $\max_r |\bB^r| \epsilon$; furthermore its objective value equals $\hat v$.\end{LE} 

\begin{LE}\label{lem:opttoensemble} Suppose $(x^*, z^*)$ is a feasible solution to problem \eqref{QuadBlockInd}.  Then there is a feasible sketch ensemble whose value is at most the value of $(x^*, z^*)$ in \eqref{QuadBlockInd}. \end{LE}
\noindent {\em Proof.} We choose a sketch for each block as follows.  For every $B_i$ with  $z^*_i = 1$ we set $\delta^i = 1$, and for $r \in \cM_i$ we choose $k^i_r$ {\bf smallest}, such that 
$$ x^*[i]^T Q^r_i x^*[i]  \ + \  \sum_{j \in B_i} d^r_j x^{*2}_j \ + \ \sum_{j \in B_i} c^r_j x^*_j \ + \  v^r_i \ \le \ \epsilon (k^i_r + 1).$$
Finally, if $z^*_i = 0$ we choose $\delta^i = 0$ and $k^i_r = 0$ for all $r$.  With these choices, feasibility of $(x^*,z^*)$ and Assumption \ref{as:scaling} imply (a) and (b) of Def. \eqref{def:sketchensemble}.
\qed
\begin{CO} A minimum-value sketch ensemble yields a solution for an instance of QIB that  (a) is $\max_r |\bB^r| \epsilon$-feasible over the mixed-integer constraints, (b) satisfies all combinatorial constraints, and (c) has value no larger than the optimal objective for QIB. \end{CO}
Our dynamic-programming algorithm, given below, solves the minimum-value sketch ensemble problem in polynomial time for fixed $\epsilon$ and $\omega$.

\subsection{Dynamic programming algorithm for minimum-value sketch ensemble problem}\label{sec:dynprogalgorithm}

\noindent In what follows we denote by $G$ the constraint-block intersection graph of an instance of problem \eqref{QuadBlockInd}.  Without loss of generality $G$ is connected.  We are given a tree-decomposition $(T, X)$ of $G$ of width $\omega$.  For ease of notation, we refer to the vertices of $T$ as \textit{nodes}.   We view $T$ as rooted at an arbitrary non-leaf node.  

\noindent {\bf Comment.} Each block $B_i$  corresponds to a clique $K_i$ of $G$ and thus (as is well-known) the set of nodes $t$ such that $B_i \subseteq X_t$ form a subtree of $T$. 
\begin{LE}\label{lem:treestruct}  The following hold without loss of generality. (a) There is a one-to-one correspondence between leaves of $T$ and blocks. Namely, for each block $B_i$ there is a leaf $t(i)$ of $T$ such that $X_t = K_i$ and conversely, for each leaf $t$ we have that $X_t = K_i$ for some block $i$. (b) For each vertex $r$ of $G$, each leaf of the subtree $\{ t \in T \, : \, r \in X_t \}$ is a leaf
of $T$. (c) Every node of $T$ has degree $ \le 2$. \end{LE}
\begin{proof}
(a) Take any node $t \in T$ such that $K_i \subseteq X_t$.  Then expand $T$ (and the tree-decomposition) by adding an edge $\{ q, t \}$ where $q$ is a new vertex, and set $X_q = K_i$. Let us call this the leaf \textit{designated} for block $B_i$ This proves the first part of (a).  For the converse, if there is a leaf $t$ which is not a designated leaf, then removing $t$ from $T$ yields a (smaller) tree decomposition by the first part of (a). (b) Suppose the claim is not true for some $r$ and let $u$ be a leaf of $\{ t \in T \, : \, r \in X_t \}$ which is not a leaf of $t$.  Then 
resetting $X_u \leftarrow X_u - \{r\}$ yields a smaller tree-decomposition of $G$ by (a). (c) If $t$ is a node of degree larger than $2$, then split $t$ into a subtree each of whose non-leaf nodes has degree $3$; each such node
$v$ is assigned $X_v = X_t$. \qed 
\end{proof}
 \noindent {\bf Comment.}  It is possible that  two distinct blocks $B_i, \, B_j$ give rise to the same clique in $G$, i.e., $K_i = K_j$.  However, since the blocks are distinct the corresponding set of variables are different.  The corollary highlights that the leaves of $T$ 
 partition the set of variables (i.e., into the blocks).
 
\vspace{.1in}
  \noindent Next we provide a tree-driven generalization for the computation of a minimum-value sketch ensemble that is suitable for dynamic programming.  We will rely on the following notation, given a node $t \in T$: 
  \begin{itemize}
      \item \boldmath $T_t$ \unboldmath is the subtree of $T$ containing all descendants of $t$. 
      \item We say that a {\bf block} $B_i$ is {\bf contained} in $T_t$ if $\{ q \in T \, : \, K_i \subseteq X_q \} \subseteq T_t$. Recall that
      $K_i$ is the clique induced by $B_i$ in $G$.
      \item We write  \boldmath $\beta_t $\unboldmath $\doteq$ set of blocks contained in $T_t$. 
      \item We say that a {\bf constraint} $r$ is {\bf contained} in $T_t$ if $\{ q \in T \, : \, r \in X_q \} \subseteq T_t$.  
      \item For every node $t$, \boldmath $\bM_t$ \unboldmath and \boldmath $\bC_t$ \unboldmath are the sets of mixed-integer 
and combinatorial constraints $r \in X_t$ (resp.).
  \end{itemize} 
We will develop, below, a dynamic-programming algorithm for the minimum-value feasible sketch ensemble problem.  The algorithm will proceed 'up' the tree $T$; this will require an appropriate generalization of the feasible sketch ensemble problem given in Definition \ref{def:partialensemble}. First we prove the following strengthening of Lemma \ref{lem:opttoensemble}. 
  \begin{LE}\label{lem:opttopartialensembl2} Let $(x^*, z^*)$ be a feasible solution to problem \eqref{QuadBlockInd}. Then there is a feasible sketch ensemble $\{ ( \delta^i, (k^i_r \, : \, r \in \cM_i)) \, : \, i \in B \}$  such that at every node $t$
\begin{subequations} \label{eq:upthetree}
    \begin{align}
        & \sum_{i \in \beta_t} k^i_r \le \lceil \epsilon^{-1} \rceil -1 \quad \forall r \in \beta_t \label{eq:upM}
    \end{align}
\end{subequations}
  \end{LE}
\noindent {\em Proof.} We choose a sketch for each block $B_i$ as follows. 
 If $z^*_i = 0$ we use the zero-block. Otherwise if $z^*_i = 1$ we set $\delta^i = 1$, and for $r \in \cM_i$ we choose $k^i_r$  smallest such that 
 \begin{align} \label{eq:crucial}
& x^*[i]^T Q^r_i x^*[i]  \ + \  \sum_{j \in B_i} d^r_j x^{*2}_j \ + \ \sum_{j \in B_i} c^r_j x^*_j \ + \  v^r_i \ \le \ \epsilon (k^i_r + 1),
\end{align}
as in the proof of Lemma \ref{lem:opttoensemble}. Hence, for any node $t$,
\begin{align} \label{eq:crucial2}
&  \sum_{i \in \beta_t } \left(x^*[i]^T Q^r_i x^*[i]  \, + \,  \sum_{j \in B_i} d^r_j x^{*2}_j \ + \ \sum_{j \in B_i} c^r_j x^*_j \, + \,  v^r_i z^*_i\right) \, \le \, \sum_{i \in \beta_t} \epsilon (k^i_r + 1).
\end{align}
The left sum has absolute value at most $1$, by Assumption \ref{as:scaling}.  Hence \ref{eq:upM} holds.  
\qed

\begin{DE}\label{def:partialensemble} Let $t \in V(T)$.  The {\bf partial sketch ensemble problem} corresponding to the pair of sets  
$\Pi = ( \{ \hat k_r  \, \in \bZ \, : \, r \in \bM_t \}, \{ \hat \kappa_r \, \in \bZ \, : \, r \in \bC_t\})$ satisfying
\begin{subequations}\label{eq:partialenum}
\begin{align} 
\forall r \in \bM_t: \   &  -\lceil \epsilon ^{-1} \rceil \, \le \, k_r \, \le \, \lceil \epsilon^{-1} \rceil - 1   \label{eq:rangek} \\ 
\forall r \in \bC_t: \  &  -|\{ i \in \bB^r \cap \beta_t \, : \, v^r_i = -1 \}|  \le \hat \kappa_r \le  |\{ i \in \bB^r \cap \beta_t \, : \, v^r_i = +1 \}| \label{eq:rangekappa}
\end{align}
\end{subequations}
 is the optimization problem given next: 

 \noindent For each block $B_i \in \beta_t$ select a feasible sketch $( \delta^i, (k^i_r \, : \, r \in \cM_i))$, so as to minimize the sum of values of the selected sketches, and subject to constraints (a)-(d).
\begin{itemize} 
\item [(a)] For each mixed-integer constraint $r$ contained in  $T_t$, $\sum_{i \in \beta_t  } \epsilon (k^i_r + 1) \, \le \, b^r + \epsilon | \bB^r |$. (Equivalently, $\sum_{i \in \beta_t  } \epsilon k^i_r \, \le \, b^r$)
\item [(b)] For each combinatorial constraint $r$ contained in $T_t$,     $\sum_{i  \in \beta_t } v^r_i  \delta^i  \ \le \ b_r$,
\item [(c)]  For each mixed-integer constraint $r \in X_t$,  $\sum_{ i \in \beta_t} k^i_r \ = \ \hat k_r$.
\item [(d)] For each combinatorial constraint $r \in X_t$,   $\sum_{i \in \beta_t }  \delta^i \ = \ \hat \kappa_r$. 
\end{itemize}
\end{DE} 
\noindent {\bf Comments.} This definition provides a tree-driven generalization for the minimum-value sketch ensemble problem given below. In this regard,
when node $t$ is the root then every constraint is contained in $T_t$ and conditions (a) and (b) become (a) and (b) of Definition \ref{def:sketchensemble}, i.e., the partial sketch ensemble problem becomes the sketch ensemble problem. Note that Lemma \ref{lem:opttopartialensembl2}
justifies the restriction \eqref{eq:rangek} (while \eqref{eq:rangekappa} is trivial).  This observation motivates our dynamic-programming algorithm.  The algorithm will move ``up'' the tree; each element of the  so-called ``state space'' at any node $t$ will the set of feasible partial sketch ensemble problems.  

Constraints (a) and (b) imply that the sketches we have selected yield an $\epsilon$-feasible (resp., feasible) solution for the mixed-integer and combinatorial constraints that are contained in $T_t$.  Constraints (c) and (d) instead 
say that the sketches match the ``partial'' solution given by the target values $\hat k_r$ and $\hat \kappa_r$ in \eqref{eq:rangek} and \eqref{eq:rangekappa}, respectively.  

{\bf Note} that, e.g., if a mixed-integer constraint  $r \in X_t$ is contained in $T_t$ then it must be the case that $\hat k_r \le b^r $ or else constraints (a) and (c) are incompatible.  Similarly with (b) and (d).  In what follows we assume we only enumerate pairs $\Pi$ that are {\bf consistent} in this sense.  
\begin{DE}\label{def:zeta} For every node $t$,   $\bR_t  \, = \, \bM_t \cup \bC_t$.  Given a pair $\Pi = ( \{ \hat k_r  \, : \, r \in \bM_t \}, \{ \hat \kappa_r \, : \, r \in \bC_t\})$ we denote by  \boldmath$\zeta_t(\Pi)$ the value of the optimization problem in Definition \ref{def:partialensemble}. \end{DE}

\noindent {\bf Initialization of algorithm: processing leaves.}  Let $t$ be a leaf of $T$; suppose it is designated for block $B_i$, i.e., $K_i = X_t$. Then the set of all (feasible) partial sketch ensemble problems at $t$ is included among the set of feasible sketches of block $B_i$. \\
\noindent Next, we proceed up the tree.  Let $t \in T$ be any node such that all its children have already been 
processed in the procedure. \\
\noindent {\bf Case 1. $t$ only has one child, $s$}.  We obtain the state space at $t$ using the following two-step procedure.  By Lemma \ref{lem:treestruct} $X_t \subseteq X_s$.  We begin by restricting each element of the state space at $s$; if
$( \{ \hat k_r  \, : \, r \in \bM_s \}, \{ \hat \kappa_r \, : \, r \in \bC_s\})$ is such an element then we simply keep the entries for those indices $r \in \bR_t$.  However such an element of the state space at $s$ may not be consistent: this is potentially the case when there is a constraint $r$ which is contained in $T_t$ but not contained in $T_s$ (i.e., the requirement (c) or (d) is not satisfied by $r$).  In that case
$( \{ \hat k_r  \, : \, r \in \cM_s \}, \{ \hat \kappa_r \, : \, r \in \bC_s\})$ should not be included in the state space at $t$.  Formally, we can state the following:

\begin{LE} Suppose that the pair $\Pi \ = \ ( \{ \hat k_r  \, : \, r \in \bM_t \}, \{ \hat \kappa_r \, : \, r \in \bC_t\})$ is consistent.  Then 
\begin{subequations} 
\begin{align} 
\zeta_t(\Pi) \ = \ & \min \zeta_s( \{ \bar k_r  \, : \, r \in \bM_s \}, \{ \bar \kappa_r \, : \, r \in \bC_s\} ) \nonumber \\
\text{subject to:} & \ \text{ $\{ \bar k_r  \, : \, r \in \bM_s \}, \{ \bar \kappa_r \, : \, r \in \bC_s\} $ is consistent  }, \nonumber\\
& \ \bar k_r = \hat k_r \ \forall \ r \in \bM_t; \quad \bar \kappa_r = \hat \kappa_r \ \forall \ r \in \bC_t. \nonumber
  \end{align}
\end{subequations}
\end{LE}
\noindent {\em Proof.} Follows from Definition \ref{def:partialensemble}. \qed
\noindent {\bf Case 2. $t$ has two children, $p$ and $q$.} In this case $X_t \subseteq X_p \cup X_q$.  Formally we now obtain, with a similar proof:

\begin{LE} \label{lem:case2} For any consistent pair $\Pi \ = \ ( \{ \hat k_r  : \, r \in \bM_t \}, \{ \hat \kappa_r \, : \, r \in \bC_t\})$,  
\begin{subequations} 
\begin{align} 
\zeta_t(\Pi) = \ & \min \zeta_p( \{ \bar k_r \, : \, r \in \bM_p \}, \{ \bar \kappa_r \, : \, r \in \bC_p \} )   +  \zeta_q( \{ \tilde k_r \, : \, r \in \bM_q \}, \{\tilde \kappa_r\, : \, r \in \bC_q\} )\nonumber\\
\text{subject to:} & \ \text{ $\{ \bar k_r  \, : \, r \in \bM_p \}, \{ \bar \kappa_r \, : \, r \in \bC_p\} $ is consistent  } \nonumber\\  
& \ \ \hat k_r = \bar k_r + \tilde k_r, \ \forall r \in M_t \, \text{s.t.} \, r \in M_p \cap M_q; \nonumber \\
& \ \ \hat \kappa_r = \bar \kappa_r + \tilde \kappa_r,  \ \forall r \in C_t \, \text{s.t.} \, r \in C_p \cap C_q \nonumber \\
& \ \ \hat k_r = \bar k_r  \ \forall r \in M_t  - M_q;  \ \hat k_r = \tilde k_r  \ \forall r \in M_t  - M_p \nonumber \\
& \ \ \hat \kappa_r = \bar \kappa_r  \ \forall r \in C_t  - C_q;  \ \hat \kappa_r = \tilde k_r  \ \forall r \in C_t  - C_p.
\end{align}
\end{subequations} 
\end{LE}
\noindent {\em Proof.} Again, follows from Definition \ref{def:partialensemble}. \qed

\noindent Now we turn to complexity analysis.  First, we have:
\begin{LE} Let $\omega$ be the treewidth of the constraint-block incidence graph.  Then the number of sketches for a given block $B_i$ is at most $O(2/ \epsilon)^{\omega + 1}$. \end{LE}

\noindent Let $\kappa$ be the largest support of any combinatorial constraint and $m_c =$ number of combinatorial constraints.  
\begin{LE}  Let $N$ be the number of nodes in the tree $T$.  The DP algorithm correctly solves the minimum-weight sketch ensemble problem in time $$O(N   (2/\epsilon)^{2\omega + 2} \kappa^{2 \min\{ m_c, \omega + 1 \}}     \,);$$ plus the time needed to solve  $O( |B| (2/\epsilon)^{\omega} )$ sketches. \end{LE}
\noindent {\em Proof.} Correctness follows from the above analysis.  Next, note that for any node $t$,   $|X_t| \le \omega + 1$.  Hence the number of pairs $\Pi$ that satisfy \eqref{eq:partialenum} in Definition \ref{def:partialensemble}  is at most 
$(2/\epsilon)^{|\bM_t|}(\max_{r \in \bC_t}|\bB^r|)^{|\bC_t|} \, \le \,(2/\epsilon)^{|\bM_t|}\kappa^{|\bC_t|} \le (2/\epsilon)^{\omega + 1} \kappa^{\min\{ m_c, \omega + 1 \}}$.     Case 1 can be handled by enumerating all pairs at node $s$ (and keeping track of minima on the fly); similarly Case 2 can be handled by enumerating combinations of pairs at nodes $p$ and $q$ respectively. Finally, processing the leaves the solution of $O( |B| (2/\epsilon)^{\omega} )$ sketches.  \qed 
\hspace{.1in}\\
\noindent  It is folklore that a tree-decomposition of a graph $G$, of any given width, only requires a tree of size linear in the number of vertices of $G$.  Here we start with $m$ nodes (one for each constraint); that is prior to adding leaves corresponding to blocks, and prior to reducing node degrees to at most three.  The former adds up to $|B|$ nodes and the latter increases the node count by a constant factor.  Thus $N = O(m + |B|)$ nodes.

\newpage

\section{NP-hardness}\label{sec:nphard}
The dynamic-programming algorithm given above is only of pseudopolynomial complexity even for fixed $\omega$.  Here we show that even highly restricted versions of the problem \eqref{QuadBlockInd} are already (weakly) NP-hard. Note that the literature contains a number of proofs of NP-hardness for cardinality-constrained optimization problmems (see, e.g., \cite{Bienstock1996} for an old proof).  

The proofs below will rely on the \textsc{SUBSET SUM} problem: given a list of $n+1$ positive integers $a_0, a_1, \ldots, a_n$ and an integer $N \le n$, find set of indices $J$ with $|J| \le N$ s.t. $\sum_{j \in J} a_j \ = \ a_0$. \\

Given an instance of \textsc{SUBSET SUM}, where without loss of generality $a_j > 1$ for all $j$, consider the following (feasibility-only) instance of \eqref{QuadBlockInd} on $5n$ continuous variables  
$s_{1j}, s_{2j}, t_{1j}, t_{2j}, x_j$ for $1 \le j \le n$.   
\begin{subequations}\label{eq:w3blocks}
\begin{align}
& \text{for $1 \le j \le n$:} \nonumber \\
& \quad s_{1j} + s_{2j} &  \ = \ 1 \\
& \quad s_{1j} - (a_j - 1) s_{2j} &    + x_j  = \ 2 \label{eq:alpha2}\\
& \quad t_{1j} + t_{2j} & \quad  \ = \ 1 \\
&  \quad  t_{1j} - (a_j - 1) t_{2j} & + \, x_j  \ = \ 2 \label{eq:beta2} \\
& \text{and} \nonumber \\
& \quad \sum_{j = 1}^n x_j \ = \ n + a_0,  \label{eq:w3sum}\\
& \quad \text{at most $3n$ of the $s, t, x$ variables can be nonzero.} \label{eq:w3card}
\end{align}
\end{subequations}
Thus, in terms of the formal statement for problem
\eqref{QuadBlockInd}, each continuous variable constitutes a block, and we need $5n$ corresponding binary variables.   The following is a proof of Lemma \ref{le:width2}. 
\begin{LE}\label{le:staircase} If there is a feasible solution to \eqref{eq:w3blocks}-\eqref{eq:w3card}, then there is a feasible solution to the 
instance of \textsc{SUBSET SUM}; and conversely. 
\end{LE}
\noindent {\em Proof.} Suppose we have a feasible solution $s, t, x$ to \eqref{eq:w3blocks}-\eqref{eq:w3card}.  Consider an index $1 \le j \le n$.  Note that 
at least one of $s_{1j}, s_{2j}$ must be nonzero.  If only one of them is nonzero, then $x_j \neq 0$.  Likewise with $t_{1j}, t_{2j}$.  In summary,  when precisely one of 
$s_{1j}, s_{2j}$ and one of $t_{1j}, t_{2j}$ is nonzero we count exactly three nonzeros, and otherwise $s_{1j}, s_{2j}, t_{1j}, t_{2j}$ are all nonzero, i.e., we count
at least four nonzeros.

Let $n_4$ be the number of indices $j$ such that we count at least four nonzeros.  It follows that the total number of nonzeros in the 
given solution to \eqref{eq:w3blocks} is at least $3n + n_4$.  By
\eqref{eq:w3card} we therefore have $n_4 = 0$. Hence, as summarized,
either $s_{1j} = 1$ and $s_{2j} = 0$ or viceversa.  In the former
case $x_j = 1$ while in the latter, $x_j = w_j + 1 > 1$. So
$$ \sum_{j = 1}^n x_j \ = \ n \, + \, \sum_{j \, : \, s_{2j} = 1} w_j$$
which equals $n + a_0$ by \eqref{eq:w3sum}.  Hence the set $\{j \, : \, s_{2j} = 1\}$ constitutes a solution to the instance of \textsc{SUBSET SUM}. 

For the converse, assume that $J \subseteq \{1, \ldots, n\}$ is feasible for the instance of \textsc{SUBSET SUM}.  Then, whenever
$j \in J$ we set $s_{1j} = t_{1j} = 0$, $s_{2j} = t_{2j} = 1$ and $x_j = w_j + 1$.  And whenever $j \notin J$ we set $s_{1j} = t_{1j} = 1$, $s_{2j} = t_{2j} = 0$ and $x_j =1$. \qed 

This result shows that problem \eqref{eq:w3blocks}  is at least as hard as \textsc{SUBSET SUM}.  Note, moreover, that the constraint-block intersection graph for this system has treewidth (in fact, pathwidth) $2$ (constraints \eqref{eq:alpha2}, \eqref{eq:beta2} and \eqref{eq:w3sum} form a clique).  The case with pathwidth $1$ remains open.

\hspace{.1in}\\
Consider, next, the following special case of QIB. We are given a {\boldmath$2\times n$ \unboldmath} matrix $A$, a vector $b \in \bR^2$, an integer $N \le n$ and values $d_j \ge 0$ for $1 \le j \le n$.  The problem is:
\begin{subequations}\label{2-row}
    \begin{align}
\text{2ROW:} \quad & \min \sum_j d_j x^2_j \\
\text{s.t.}  & \quad  \sum_{j = 1}^n a_{ij} x_j \, = \, b_i, \ \ i=1,2, \\
& \quad \text{at most $N$ $x_j$ are nonzero.}
\end{align}
\end{subequations}
This is an example of \eqref{QuadBlockInd} with $m = 2$ constraints; we will show next that \eqref{2-row} is weakly NP-hard. \footnote{The reader may notice the absence of bounds on the $x$ variables. However, the proof shows that $0 \le x_j \le 1$ for all $j$ can be assumed.} 

Given an instance of \textsc{SUBSET-SUM}, consider the following instance of 2ROW:
\begin{subequations} \label{eq:2rowproof}
\begin{align}
   V^* \ \doteq \ \min \; & \ \sum_{j = 1}^n  x_{j}^2  \\
    \text{s.t. } & \sum_{j = 1}^n x_{j} = N \label{eq:selector} \\
    & \ \sum_{j = 1}^n a_{j} x_{j} = a_0 \label{eq:suma}\\
    & \ \exists \, J \subseteq\{1, \ldots, n\} \ \text{with $|J| \le N$ s.t.} \ x_{j} = 0 \ \text{if} \ j \notin J. \label{eq:subset}
\end{align}
\end{subequations}

\noindent The next result is a proof of Lemma \ref{le:2rowcard}.
\begin{LE}\label{le:width2linear}
There is a subset $J\subseteq \{1, \ldots, n\}$ with $|J| = N$ such that $\sum_{j \in J} a_j = a_0$ if and only if $V^* \le N$. 
\end{LE}
\noindent {\em Proof.}      If such a set $J$ exists, then set $x_{j}  = 1$ for $j \in J$ and $x_{j} = 0$ otherwise. This solution is feasible for \eqref{eq:2rowproof} and yields an objective of $N$.  
     
     For the converse, assume $V^* \le N$ and let $(x^*, J^*)$ denote an optimal solution to \eqref{eq:2rowproof}. Note   that  the optimal solution for a problem of the form $\min\{\sum_{j \in S} y_j^2: \sum_{j \in S} y_j = 1\}$  is uniquely attained by setting $y_j = 1/|S|$ for all $j \in S$.  Hence \eqref{eq:selector} and \eqref{eq:subset} imply that 
     $\sum_{j = 1}^n x^{*2} \ge N$.  Thus, $V^* \le N$ implies $V^* = N$, $|J^*| = N$ and $x^*_j = 1$ for all $j \in J^*$. \qed

%
%
\section{Glossary}
\label{sec:glossary}
\noindent  \boldmath $B$. \unboldmath The set of blocks.\\
\noindent \boldmath $\bB^r$. \unboldmath  The set of blocks incident with a constraint $r$.  Definition \ref{def:incidence}.\\
\noindent \boldmath $\beta_t$. \unboldmath  In a tree decomposition $(T,X)$, the set of blocks contained in $T_t$. \\ 
\noindent \boldmath $\bC_t$. \unboldmath In a tree decomposition $(T,X)$, the set of combinatorial constraints contained in a bag $X_t$.\\
\noindent \boldmath $G$. \unboldmath The constraint-block intersection graph.\\
\noindent \boldmath $\kappa$. \unboldmath The largest support of any combinatorial constraint.\\
\noindent \boldmath $K_i$.  \unboldmath The clique of the constraint-block intersection graph induced by block $B_i$.\\
\noindent \boldmath $m_c$. \unboldmath The number of combinatorial constraints.\\
\noindent \boldmath $\cM_i$. \unboldmath The set of mixed-integer constraints incident with block $B_i$.  \\
\noindent \boldmath $\bM_t$. \unboldmath In a tree decomposition $(T,X)$, the set of mixed-integer constraints contained in a bag $X_t$.\\
\noindent {\bf node}.  A vertex of the tree $T$ used in a tree-decomposition.\\
\noindent \boldmath $\Pi$. \unboldmath A pair
$( \{ \hat k_r  \, \in \bZ \, : \, r \in \bM_t \}, \{ \hat \kappa_r \, \in \bZ \, : \, r \in \bC_t\})$ as in Definition \ref{def:partialensemble}.  Consistency will be applied (see the definition).\\
\noindent \boldmath $\bR_t$. \unboldmath In a tree decomposition $(T,X)$, $\bC_t \cup \bR_t$.\\
\noindent \boldmath $T_t$.  \unboldmath The subtree of the (rooted) tree $T$ that is rooted at node $t$.\\
\noindent \boldmath $X_t$.  \unboldmath In a tree-decomposition $(T,X)$, the set of vertices (``bag'') corresponding to node $t$.\\
\noindent \boldmath $\zeta_t$. \unboldmath Given a  pair $\Pi = ( \{ \hat k_r  \, : \, r \in \bM_t \}, \{ \hat \kappa_r \, : \, r \in \bC_t\})$,  $\zeta_t(\Pi)$  is the value of the optimization problem in Definition \ref{def:partialensemble}.  Definition \ref{def:zeta}.\\ \\
\noindent {\bf containment}. A block $B_i$ is contained in $T_t$ if the set of nodes $q$ of $T$ s.t. $B_i \subseteq X_q$ is contained in $T_t$.  A constraint $r$ is contained in $T_t$ if the set of nodes $q$ of $T$ s.t. $r \in X_q$ is contained in $T_t$.

\newpage

\bibliographystyle{splncs04.bst}
\bibliography{references}

\begin{thebibliography}{10}
\providecommand{\url}[1]{\texttt{#1}}
\providecommand{\urlprefix}{URL }
\providecommand{\doi}[1]{https://doi.org/#1}

\bibitem{zhourosvany}
Ben-Tal, A., Bends\o{}e, M.P.: A new method for optimal truss topology design.
  SIAM Journal on Optimization  \textbf{3}(2),  322--358 (1993).
  \doi{10.1137/0803015}, \url{https://doi.org/10.1137/0803015}

\bibitem{BertsimasRyan}
Bertsimas, D., Cory-Wright, R.: A scalable algorithm for sparse portfolio
  selection. Informs J. Computing  \textbf{34},  1489--1511 (2022)

\bibitem{Bienstock1996}
Bienstock, D.: Computational study of a family of mixed-integer quadratic
  programming problems. Mathematical Programming  \textbf{74},  121--140 (1996)

\bibitem{BienstockMunoz2018}
Bienstock, D., Mu\~noz, G.: {LP Formulations for Polynomial Optimization
  Problems}. {SIAM} J. Optimization  \textbf{28}(2),  1121--1150 (2018)

\bibitem{bodlaender}
Bodlaender, H.L.: Dynamic programming on graphs with bounded treewidth. In:
  Lepist\"{o}, T., Salomaa, A. (eds.) Automata, Languages and Programming,
  Lecture Notes in Computer Science, vol.~317, pp. 105--118. Springer Berlin
  Heidelberg (1988)

\bibitem{qplib}
Furini, F., Traversi, E., Belotti, P., Frangioni, A., Gleixner, A., Gould, N.,
  Liberti, L., Lodi, A., Misener, R., Mittelmann, H., Sahinidis, N., Vigerske,
  S., Wiegele, A.: Qplib: a library of quadratic programming instances.
  Mathematical Programming Computation  \textbf{11}(2),  237--265 (Jun 2019).
  \doi{10.1007/s12532-018-0147-4}, publisher Copyright: {\textcopyright} 2018,
  Springer-Verlag GmbH Germany, part of Springer Nature and The Mathematical
  Programming Society.

\bibitem{GolubVanLoan}
Golub, G.H., Van~Loan, C.F.: Matrix computations, vol.~3. JHU Press (2012)

\bibitem{Gomez2024}
G\'{o}mez, A., Han, S., Lozano, L.: Real-time solution of quadratic
  optimization problems with banded matrices and indicator variables.
  arXiv:2405.03051 pp. 1--30 (2024)

\bibitem{perspective}
G\"{u}nl\"{u}k, O., Linderoth, J.: Perspective reformulations of mixed integer
  nonlinear programs with indicator variables. Mathematical Programming
  \textbf{124},  183–205 (2010)

\bibitem{Schrijver}
Schrijver, A.: Theory of Linear and Integer Programming. John Wiley \& Sons,
  Inc., New York, NY, USA (1986)

\end{thebibliography}

\end{document}